\documentclass[11pt]{amsart}

\usepackage[margin=1in, marginparsep=1pt,
            marginparwidth=124pt]{geometry}

\usepackage[english]{babel}
\usepackage[usenames,dvipsnames,svgnames,table]{xcolor}
\usepackage[pagebackref,linktocpage=true,colorlinks=true,linkcolor=Blue,citecolor=BrickRed,urlcolor=RoyalBlue]{hyperref}
\usepackage[alphabetic]{amsrefs}
\usepackage{skak} 

\usepackage[colorinlistoftodos,prependcaption,textsize=tiny]{todonotes}
\usepackage{amsmath,amssymb,amsfonts,amsthm,enumerate,textcomp}
\usepackage{url}
\usepackage{graphicx}
\usepackage{xcolor}
\usepackage{wrapfig}

\usepackage{graphicx,stackengine,scalerel}
\def\tang{\ThisStyle{\abovebaseline[0pt]{\scalebox{-1}{$\SavedStyle\perp$}}}}

\usepackage{mathrsfs}

\usepackage{esint}
\usepackage{enumitem}

\theoremstyle{plain}
\newtheorem{thm}{Theorem}[section]
\newtheorem{defn}[thm]{Definition}
\newtheorem{lemma}[thm]{Lemma}
\newtheorem{prop}[thm]{Proposition}
\newtheorem{remark}[thm]{Remark}

\newtheorem*{m-thm}{Main Theorem}
\usepackage{amsmath}
\usepackage{amssymb}

\numberwithin{equation}{section}

\newcommand\R{\mathbb R} 
 
\renewcommand\S{\mathbb S} 

\newcommand{\diam}{{ \rm diam}}

\newcommand{\dist}{{ \rm dist}}
\newcommand\p{\partial}

\newcommand\e{\varepsilon}
\renewcommand{\epsilon}{\e}
\renewcommand{\div}{\mbox{div}}

\renewcommand\H{\mathscr H}
\newcommand\F{\mathscr F}

\newcommand{\I}[1]{\chi_{\{#1>0\}}}
\newcommand{\po}[1]{\{#1>0\}}
\newcommand{\fb}[1]{\partial\{#1>0\}}

\newcommand{\fbr}[1]{\partial_{\text{red}}\{#1>0\}}

\newcommand\Om{\Omega}
\newcommand\na{\nabla}
\newcommand\HD{{\rm{HD}}} 
\usepackage{tikz}
\usetikzlibrary{arrows,shapes,positioning}
\usetikzlibrary{decorations.markings}
\tikzstyle arrowstyle=[scale=1]
\tikzstyle directed=[postaction={decorate,decoration={markings,
    mark=at position .65 with {\arrow[arrowstyle]{stealth}}}}]
\tikzstyle reverse directed=[postaction={decorate,decoration={markings,
    mark=at position .65 with {\arrowreversed[arrowstyle]{stealth};}}}]

\allowdisplaybreaks


\begin{document}

\title[Full and partial regularity for a class of nonlinear free boundary problems]{Full and partial regularity for a class of nonlinear free boundary problems}

\author{Aram L. Karakhanyan}
\address{Aram Karakhanyan:
School of Mathematics, The University of Edinburgh,
Peter Tait Guthrie Road, EH9 3FD Edinburgh, UK}
\email{aram6k@gmail.com}

\begin{abstract}
In this paper we classify the nonnegative global minimizers of the functional
\[
J_F(u)=\int_\Om F(|\na u|^2)+\lambda^2\I u, 
\]
where $F$ satisfies some structural conditions and $\chi_D$ is the characteristic function of 
a set $D\subset \R^n$. We compute the second variation of the energy and study the properties of the stability operator. 
The free boundary $\fb u$ can be seen as a rectifiable $n-1$ varifold. If the free boundary is a Lipschitz multigraph then 
we show that the  
first variation of this varifold is bounded and  use Allard's monotonicity formula to prove the existence of 
 tangent cones modulo a set of small Hausdorff dimension. In particular we prove that if $n=3$
and the ellipticity constants of the quasilinear elliptic operator generated by $F$ are close to 1 then the conical free boundary must be flat. 
\end{abstract}

\maketitle

\section{Introduction}

In this paper we consider the regularity of the global minimizers of the functional

\begin{equation}\label{eq:J-quasi}
J_F(u)=\int_\Om F(|\na u|^2)+\lambda^2\I u \to \min 
\end{equation}
over the class of admissible functions 
\[
 u\in \mathcal A=\{u\in W^{1, F}(\Om), u-u_0\in W^{1, F}_0(\Om)\}
\]
with $F\in C^{2, 1}[0, \infty)$
satisfying the structural  conditions
\begin{equation}\label{eq:F-cond}
c_0\le F'(t)\le C_0, \quad 0\le F''(t)\le \frac{C_0}{1+t}, 
\end{equation} 
for some positive constants $c_0, C_0$. Here $\Om\subset \R^n$ is a domain, $\lambda>0$ given constant, 
$W^{1, F}$ the Orlicz-Sobolev space generated by $F$, and $u_0\in W^{1, F}(\Om)$ given Dirichlet datum. 

If the free boundary $\fb u$ is smooth then $\na u$
satisfies the following implicit Bernoulli type condition
\begin{equation}
f(\na u)+\lambda^2-\na_\xi f(\na u)\na u=0, \quad f(\xi)=F(|\xi|^2).
\end{equation}
One can deduce that $|\na u|=\lambda^*$ on the free boundary, where 
$\lambda^*$ is determined from the implicit relation 
$$
\lambda^2=2F'(|\na u|^2)|\na u|^2-F(|\na u|^2).
$$
Throughout  this paper we assume $\lambda^2=2F'(1)-F(1)$ so that $|\na u|=1$ on $\fb u$, see Remark \ref{rem:Bernoulli}.

\begin{m-thm}
Let $u\ge 0$ be a global minimizer of \eqref{eq:J-quasi} with 
$F$ satisfying \eqref{eq:F-cond}.

\begin{itemize}
\item[(a)] If $n=3$ and $\fb u$ is a cone and 
\begin{eqnarray}\label{small-ellipt}
\sup F''(|\na u|^2)<\frac32{F'(1)}
\end{eqnarray}
then $u(x)=\max(x_1, 0)$ in some coordinate system.
\item[(b)] 
Suppose that
the rectifiable $n-1$ varifold $V=(\Gamma, \theta)$ associated with $\Gamma=\fb u$ 
has bounded first variation. 
If for some $\alpha\in(0, 1)$ and $\xi\in \Gamma$ 
\begin{equation}
\int_{B_\rho(\xi)}|H|\le  C(\xi, \alpha) r^{n-2+\alpha}, \quad r\in (0, R), \overline{B_R}(\xi)\subset U\subset \R^n, 
\end{equation} 
$n=3$ and 
\eqref{small-ellipt}
holds 
then $\Gamma$ is regular at $\xi$.
Moreover, the singular set has Hausdorff dimension $\le 1$. 
\end{itemize}

\end{m-thm}

The proof follows from the combination of Theorems \ref{thm:Bernstein} and \ref{thm:partial} proved below. When $F(t)=t^{\frac p2}, 2<p<5$ then our argument shows that 
\eqref{small-ellipt} is satisfied and the homogeneous free boundaries must be flat in $\R^3$. 
One of the open questions in the free boundary regularity theory asks whether the global minimizers of the Alt-Caffarelli functional \cite{AC} are flat in $\R^n, {n\le 7}$. 
Affirmative answer  is given for $n=2$ \cite{AC}, $n=3$ \cite{CJK}, $n=4$ \cite{JS}. For $n=7$ an analogue of the 
Simons  cone shows that non-flat global minimizers exist  for $n\ge 7$ \cite{JdS}. Our main results contributes in this direction for nonlinear $F$.

The paper is organized as follows: In section \ref{sec:2} we characterize the flat free boundary points
via the variational mean curvature. The free boundary is smooth at the flat points, see Theorem 6.1 \cite{ACF-quasi}. Our argument  shows that if the free boundary contains a smooth portion of minimal surface then it should be a plane. 

In section \ref{sec:3} we recall some well-known facts about varifolds. The main tool we need to 
prove our main theorem is the Allard monotonicity formula. For this we need to consider the first 
variation of the varifold constructed from $\Gamma$, and show that the first variation is finite for 
Lipschitz multi-graphs with nonnegative mean curvature using the 
weak convergence of mean curvature measure \cite{TW}.  
To every varifold with bounded first variation one can assign a Radon measure 
which in the case of smooth varifold agrees with the mean curvature vector, which is the 
trace of the second fundamental form.  If the relative density of this measure decays 
sufficiently fast at some point $\xi$ then it implies  existence of tangent cones. 
For the classical case \cite{AC} the existence of tangent cones follows from the  monotonicity 
formula of Weiss \cite{Weiss} which is a version of linearized radially symmetric entropy for minimal surfaces in $\R^n$.

Section \ref{sec:4} contains the main technical tool to be used in the proof 
of Main Theorem, the stability inequality \eqref{eq;2nd-var-ineq}. We follow the 
argument of \cite{CJK} closely. Our main task is to obtain 
the expansion  for the gradient term for perturbed function in the functional which  in the classical case c easily follows from divergence formula thanks to the fact that the Laplace operator is self-adjoint.   
We also discuss the related stability operator and provide some explicit computations in section 
\ref{sec:5}

In section \ref{sec:6} we prove our main theorem by combining the results we obtained in previous sections. Lastly, we prove a global result for smooth minimizers in $\R^2$ at the end of this section.  
 
 Finally, in Section \ref{sec:7} we construct a weak solution to the free boundary problem 
 such that its free boundary is a double cone. For simplicity we consider the case 
 $F(t)=t^{p/2}$. We prove that this solution is not a minimizer.  
\subsection*{Notation}
We fix some notation  be used throughout of paper: 
Let $u$ be a minimizer of \eqref{eq:J-quasi} then $\Om^+_u=\Om^+(u)=\po u$,  $\Om^+_R(u)=\Om^+(u)\cap B_R$, where $B_R$ is the open ball centered at the origin.
The free boundary is denoted by $\Gamma=\fb u$, and $\H^s$ is the $s$-dimensional 
Hausdorff measure.  $V=v(\Gamma, \theta)$ is the varifold associated with $\Gamma$
and we usually consider the portion of $\Gamma$ in some bounded subdomain $U\subset \R^n$
where $0\in \Gamma\cap U$.

\section{Variational mean curvature}\label{sec:2}
In this section we characterize the flat free boundary points in terms of small density of variational mean curvature of
$\fb u$. For every set $E$ of locally finite perimeter we can find an integrable function $H$ which is the 
variational mean curvature  of $E$ \cite{BGM}. 
\begin{lemma}
Let $u_0$ be a  global minimizer of \eqref{eq:J-quasi} then the free boundary $\fbr {u_0}$ is a generalized surface of non-positive outward mean curvature, i.e.
if $S\subset \fbr{u_0}$ and $S'\subset \po {u_0}$ such that $\p S=\p S'$ then 
\begin{equation}\label{eq:mincurv-H}
\H^{n-1}(S)\le \H^{n-1}(S').
\end{equation}
\end{lemma}
\begin{proof}
Recall that $\fbr{u_0}$ is relatively open subset of $\fb {u_0}$. Moreover, $\fbr{u_0}$
is smooth \cite{ACF-quasi}. Consequently in the domain $D$ bounded by $S$ and $S'$ 
we have   
\[
0
=
\int_D \mathcal L u_0
=
\int_D \div (2F'(|\na u_0|^2)\na u_0). 
\]
After applying the divergence theorem we get that 
\[
\lambda^*F'((\lambda^*)^2)\H^{n-1}(S)
=
\int_{S}|\na u_0|F'(|\na u_0|^2)
=
\int_{S'} F'(|\na u_0|^2)\p_\nu u_0.
\]
By Theorem 4.1 \cite{ACF-quasi} we have that $|\na u_0|\le \lambda^*$ in $\R^n$, thus we conclude from the last identity and \eqref{eq:F-cond} that 
\[
\H^{n-1}(S)\le \H^{n-1}(S').
\]
Suppose that $x_0\in S$ and choose the coordinate system near $x_0$ so that 
$x_n$ points in the direction of the outer normal to $\po {u_0}$ and $x'=(x_1, \dots, x_{n-1})$. 
Let $S$ be the graph $x_n=f(x')$ near $x_0$, where $f$ is some smooth function 
over $B^{n-1}_r(x_0)=\{x'\in \R^{n-1}\ :\ |x'-x_0'|<r\}$ for some small $r>0$.
Then \eqref{eq:mincurv-H} can be rewritten in the following equivalent form 
\[
\int_{B^{n-1}_r(x_0)}\sqrt{1+|\na_{x'} f|^2}\le \int_{B^{n-1}_r(x_0)}\sqrt{1+|\na_{x'} (f-\e\phi)|^2}
\]
for every $0\le \phi\in C_0^\infty(B^{n-1}_r(x_0))$ and $\e>0$ small. This implies that 
\[
\int_{B^{n-1}_r(x_0)}\frac{\na_{x'}f}{\sqrt{1+|\na_{x'}f|^2}}\na \phi\ge 0.
\]
Therefore $\div_{x'}\left(\frac{\na_{x'}f}{\sqrt{1+|\na_{x'}f|^2}}\right)\ge 0$ and the mean curvature of 
$S$ is nonnegative. 
\end{proof}

For every set of finite perimeter $E\subset \R^n$ it is possible to find an integrable function $H$, 
called the variational mean curvature, such that $E$ minimizes the functional 
\begin{equation}\label{eq:F-def}
\F(F, U)=\int_U |D\chi_F|+\int_{U}\chi_F(x) H(x)dx.
\end{equation}
More precisely we have 
\begin{defn}\label{def:H}
A set $E$ is said to have variational mean curvature $H$ in $U$ if 
\begin{itemize}
\item[(i)] $\int_V |D\chi_E|<\infty, \forall V\Subset U$, 
\item[(ii)] $\F(E, U)\le \F(F, U)$  $\forall V\Subset U$, $\forall F\subset U$ such that 
$(E\setminus F)\cup(F\setminus E)\Subset V$. 
\end{itemize}
\end{defn}

If $E$ is  a set of finite perimeter then we can construct $H$ as follows:
take a measurable function $h\ge0$ such that $\int_E h<\infty$ and 
$\int_F h=0$ iff $|F|=0$. For $\sigma\ge 0$  and $F\subset E$, consider the functional 
\[
\mathcal B_\sigma(F)=\int_{\R^n}|D\chi_F|+\sigma\int_{E\setminus F}h.
\]
Then for the minimizing sets we have $E_\sigma\subset E_\mu$ if $0\le \sigma<\mu$,  
and $\cup_{\sigma}E_\sigma=E$. 
By defining 
\[
H(x)=-\inf\{\sigma h(x), x\in E_\sigma, \sigma\ge 0\}, \quad \forall x\in E
\]
we obtain the desired variational mean curvature function $H$ defined on $E$.
Arguing as above we can define $H$ on $\R^n\setminus E$ analogously. 

\begin{remark}
The variational mean curvature is not unique. As the construction above shows 
it depends on the choice of the weight function $h$. 
An interesting choice is $h(x)=\dist (x, \p E)$ \cite{Almgren}.

Then for given $\sigma$ we have $E_\sigma$ with $\p E_\sigma=\Sigma_{reg}+\Sigma_s$
such that $\Sigma_{reg}$ is $C^{2, \alpha}$ surface and if we denote 
$M_\sigma=\p E_\sigma$ then $H_\sigma(x)=\sigma h(x)\nu (x)$ on regular part and 
\[
\int_{M_\sigma}\div X=-\sigma\int_{M_\sigma} h(x)X\cdot \nu.
\]
Moreover, as $\sigma\to \infty$ we have weak converges of measures 
provided that $\p E$ is a $C^{1, \alpha}$ graph with nonnegative mean curvature \cite{BGM}. 
\end{remark}

Next we want to characterize the flat points of $\fb u$ via $H$. Recall 
that $x_0\in \fb u$ is said to be $\e$ flat in $B_r(x_0)$ if 
\begin{equation}\label{blyangtrs}
\inf_\nu\{h : \partial\{u>0\}\cap B_r(x_0)\subset S(h; x_0, \nu)\cap B_r(x_0)\}<\e r, 
\end{equation}
where 
\begin{equation}
\mathcal S(h; x_0, \nu):=\{x\in \R^n : -h<(x-x_0)\cdot\nu<h\}
\end{equation}
is  the slab of height $2h$ in unit direction $\nu$. 
An equivalent form of \eqref{blyangtrs} is
\[
\HD\Big(\fb u\cap B_r(x_0), \Pi \Big)< \e\, r \quad \mbox{for some}\  \ n-1\ \mbox{dimensional plane}\ 
  \Pi\ \mbox{ passing through}\ x_0, 
\]
where 
\begin{equation*}
\HD(A, B):=\max\left\{\sup\limits_{a\in A}\dist(a, B),\;  
\sup\limits_{b\in B}\dist(b, A)\right\}
\end{equation*}
is the Hausdorff distance of two sets $A, B$. 
\begin{thm}\label{thm:mean-curv}
Let $u$ be a global minimizer of  \eqref{eq:J-quasi}. Then for every $\e$ there are $\delta>0, r_0>0$ such that 
if $\frac1{r^{n-1}}\int_{B_{r}(x_0)}|H|<\delta , r<r_0$ then $B_{r/2}(x_0)\cap \fb {u_0}$ is $\e$ flat.
\end{thm}
\begin{proof}
Suppose there exists $\e_0>0$ and sequences $\{\delta_k\}_{k=1}^\infty, \delta_k\to 0, \{r_k\}_{k=1}^\infty, r_k\to 0, \{x_k\}_{k=1}^\infty, x_k\in \fb u$ such that 

\begin{equation}
\frac1{r_k^{n-1}}\int_{B_{r_k}(x_0)}|H|<\delta_k
\end{equation}
but 
\begin{equation}\label{eq:non-flat-k}
 \mbox{HD}(\fb u\cap B_{r_k}(x_k), \Pi)\ge \e_0 r_k, \quad \forall \ n-1\ \mbox{dimensional plane}\ 
  \Pi\ \mbox{ passing through}\ x_k.
\end{equation}

Let $E_k$ be the free boundary of $u_k(y)=u(x_k+r_k y)/r_k$ and set 
$H_k(y)=r_k H(x_k+r_k y)$. Then we have 
\[
\int_{B_1}|D\chi_{E_k}|+\int_{B_1}H_k(y)\chi_{E_k}
\le
\int_{B_1}|D\chi_{F}|+\int_{B_1}H_k(y)\chi_{F} \quad (E_k\setminus F)\cup (F\setminus E_k)\Subset V.
\] 
Observe that $0\in E_k$ and $\int_{B_1}|H_k|<\delta_k$. Moreover, 
there is a subsequence $k_j\to \infty$ such that $u_{k_j}\to u_0$ uniformly in $C^{\alpha}_{loc}$, $\na u_{k_j}\to \na u_0$ weakly star in $L^\infty_{loc}$ and, moreover, 
by (3.11) \cite{ACF-quasi} $E_{k_j}=\fb {u_{k_j}}$ converge to $E_0=\fb{u_0}$ in Hausdorff distance. 
Using a customary 
 compactness argument for BV functions (e.g. Theorem 6.3 \cite{Simon}) we can extract a 
 subsequence, still labelled $E_{k_j}$ such that $\chi_{E_{k_j}}\to \chi_{E_0}$ in $BV(B_1)$.
It is easy to see that $(E_0\setminus F)\cup (F\setminus E_0)\Subset V$ thanks to the convergence of 
$E_{k_j}$ to $E_0$ in 
Hausdorff distance. Consequently 
\[
\int_{B_1}|D\chi_{E_0}|
\le
\int_{B_1}|D\chi_{F}| \quad (E_0\setminus F)\cup (F\setminus E_0)\Subset V.
\] 
Therefore $\fb {u_0}$ is a generalised minimal surface. Note that \eqref{eq:non-flat-k}
translates to 
\begin{equation}\label{eq:non-flat-0}
\mbox{HD}(\fb {u_0}\cap B_{1}(0), \Pi)\ge \e_0 , \quad \forall \ n-1\ \mbox{dimensional plane}\ \Pi \ \mbox{passing through}\ 0.
\end{equation}

If $u_0$ is the limit as above then $\fb{ u_0}$ contains a smooth piece $S$ of a minimal surface, because 
$\fbr {u_0}$ is relatively open in $\fb {u_0}$, see Theorem 6.2  \cite{ACF-quasi} . 
Let $e$ be the unit outer normal at some $y_0\in S$ such that 
$y_0\in \fbr {u_0}$. Introduce  $w(x)=1+\p_e u_0$, then differentiating 
$\mathcal L u_0=0$ in $e$ direction we get that $\div(a_{ij}\na w)=0$, where 
$a_{ij}=F'(|\na u_0|^2)\delta_{ij}+2F''(|\na u_0|^2\p_iu_0\p_ju_0)$ is a uniformly elliptic matrix thanks to 
assumptions \eqref{eq:F-cond}. Since $u_0$ is 
smooth near $y_0\in S$ we see that 
$a_{ij}w_{ij}+b_iw_i=0$ where $b_i=\sum_j \p_j a_{ij}$. From Hopf's lemma 
$\p_e w(y_0)\not =0$.  Choose the coordinate system at $y_0$ such that $e$  is pointing in $x_n$ direction. From  $|\na u_0|^2=1$ near $y_0$ on $S$ we infer that  
$u_n u_{i n}=0$ for every $i=1, \dots, n-1$. In this coordinate system the mean curvature of 
$S$ at $y_0$ is $\sum_{i=1}^{n-1}\p_{ii} u_0=0$. This in conjunction with the equation 
$a_{ij}u_{ij}=0$ yields that $u_{ee}(y_0)=u_{nn}(y_0)=0$.    
However,   $\p_n w= u_{nn}=0$ at $y_0$ and this is in contradiction with Hopf's lemma.
Thus $u_0=x_n+g(x'), x'=(x_1, \dots, x_{n-1})$ in $\po {u_0}$ for some function $g$. 
The free boundary condition implies that $|\na g|=0$ on $\fbr {u_0}$. Since 
$g$ is continuous it follows that  $g$ is constant on $\fb {u_0}$. Evidently $g$ solves the 
equation $\mathcal L g=F'(1+|\na g|^2)\Delta g+2F''(1+|\na g|^2)\na gD^2 g\na g=0$ in $\po{u_0}$.
Thus without loss of generality we can assume that  $g=|\na g|=0$ on $\fb{u_0}$
and $\mathcal L g=0$ in $\po{ u_0}$. Suppose that there is $z_0\in \fbr{u_0}$ such that 
$g\ge 0$ (or $g\le 0$) near $z_0$. Then Hopf's lemma implies that $g$ vanishes identically.  Consequently $g=0$ and $u_0=x_n^+$. But this is in contradiction with \eqref{eq:non-flat-0}.
\end{proof}

\begin{remark}
Let $H\in L^p$ be the variational mean curvature in the sense of
Definition  \ref{def:H}.
If $p=n$ then the boundary of $E$ is $C^\alpha, \alpha\in(0, 1)$ away from a set of dimension $n-7$.
If $p>n$ then $\p E$ is almost minimal. If 
$1\le p<n$ we cannot expect any regularity \cite{BGM}. 
\end{remark}
\section{Mean curvature measure}\label{sec:3}
In this section we introduce some basic facts about varifolds and 
show that under some conditions the free boundary $\Gamma$ equipped with $\H^{n-1}$ as weight measure
becomes a rectifiable $n-1$ varifold.  We also introduce Allard's monotonicity formula \cite{Allard}
in order to show the existence of tangent cones. 
\subsection{First variation of varifold}
Recall that by Theorem 6.2 \cite{ACF-quasi} $\Gamma$ is rectifiable. Then we define the 
rectifiable $n-1$ varifold $V= v(\Gamma,  \theta)$ (with multiplicity $\theta$) as the 
equivalence class of all pairs $(\widetilde\Gamma, \widetilde \theta)$ such that 
$\widetilde\Gamma$ is $n-1$ rectifiable, $\H(\Gamma\triangle\widetilde \Gamma)=0$ and 
$\theta=\widetilde \theta$ a.e. on $\Gamma\cap \widetilde\Gamma$ \cite{Simon} page 77.
Then the weight measure is $\mu_V=\H^{n-1}\with \theta$. 
We say that $V$ has bounded first variation if there is a constant $c>0$ such that 
\begin{equation*}
\sup_{X\in C^{1}_0(U), |X|\le 1}\left|\int \div Xd\mu_V\right|\le c.
\end{equation*}

By the Riesz represenation theorem there is a vector measure  $H$ such that 
\begin{equation*}
\int \div X=-\int H\cdot X.
\end{equation*}

Suppose $V$ is smooth then $X=X^{\tang}+X^\perp$,  where $X^{\tang} $  and $X^\perp$ are 
 the tangential and normal components of $X$, respectively.  
From the divergence theorem it follows that $\int\div X^{\tang}=0$. Writing $X^\perp=E n, E=X\cdot n, $ we get 
\begin{eqnarray*}
\int\div X
&=&
\int \div(X^\perp)=\int \tau_i\cdot \na_{\tau_i}(En)\\
&=&
\int\tau_i\cdot 
\left[ \na_{\tau_i}E n+E\na_{\tau_i} n\right]=\int A(\tau_i, \tau_i)(X\cdot n)\\
&=&
-\int H\cdot X,
\end{eqnarray*}
where $A$ is the second fundamental form, $\tau_1,\dots, \tau_{n-1}$ is
any orthonormal basis for the tangent space,  and $H$ is the { mean curvature} measure. 

\subsection{Lipschitz multigraphs}
The mean curvature measure can be prescribed to a class of graphs of semicontinuous functions which 
are subsolutions to the mean curvature equation in viscosity sense \cite{TW}.  
Suppose that $\{f_\alpha\}$ is a family of Lipschits continuous functions $f_\alpha:B'_1\to \R, f_\alpha(0)=0$ such that 
$|f_\alpha(x')-f(y')|\le L|x'-y'|, x', y'\in B_1'$ for some $L>0$, where 
$B_1'=\{|x'|<1\}, x=(x', x_n)$. Suppose that at  $x'\in B'_1$ we have $f_\alpha(x')=f_\beta(x')$ for some 
$\alpha\not=\beta$. In other words, $(x', f_\alpha(x'))$ is a point of self intersection of the free boundary.
Then the set of such points  has zero  $n-2$ dimensional Hausdorff measure. To prove the claim it is enough to show that 
$(x',f_\alpha(x'))\not \in \fbr u$. Otherwise there is a blow-up $u_0$ limit of $u$ at $(x',f_\alpha(x'))$ such that 
the free boundary $\fb {u_0}$ contains a plane, hence it follows from the proof of Theorem \ref{thm:mean-curv} that at $0$ the free boundary  $\fb {u_0}$ cannot have self intersection.

As a corollary we can show that the number of components of $\fb u$ near $0$ is bounded: Suppose 
$\fb u \cap Q=\cup_\alpha Graph_{f_\alpha}$ where we denote $Q=B_1'\times [-1, 1]$ such that $f_\alpha(0)=0$.
Then we have from \cite{ACF-quasi} Theorem 3.2 that there is a constant $C$ such that 
\[
C\ge \H^{n-1}(\fb u\cap Q)\ge \sum_{\alpha=1}^N\int_{B_1'}\sqrt{1+|\na_{x'}f_{\alpha}|^2}\ge N \H ^{n-2}(B_1').
\]  
Hence $N$ is bounded.
\medskip

There is a subtle approximation argument 
\cite{TW} Theorem 5.1 and Remark 5.1 that allows to construct 
an approximating sequence $f^k_\alpha$ for which 
$\mathcal Mf^k_\alpha\ge \sigma_k$ and $\sigma_k\to 0$ 
boundedly. 

Using this argument and weak continuity of 
the first variations $\delta V_k$ for $f^k_\alpha$  
we can show that $V$ has bounded first variation. Indeed, we have

\begin{eqnarray*}
\left|\int_{Graph_{f^k_\alpha}}\div X\right|
&=&
\left|\int_{B_1'}\sqrt{1+|\na f_\alpha^k|^2} \mathcal Mf^k_\alpha |X|dx'\right|\\
&\le &
\left|\int_{B_1'}\sqrt{1+|\na f_\alpha^k|^2} (\mathcal Mf^k_\alpha -\delta_k)|X|dx'+
\delta_k \int_{B_1'}\sqrt{1+|\na f_\alpha^k|^2} |X|dx'\right|\\
&\le &
\sqrt{1+L^2}\int_{B_1'}(\mathcal Mf^k_\alpha -\delta_k)dx'+
\delta_k \H^{n-1}(\fb u\cap Q)\\
&=&
\sqrt{1+L^2} \int_{\p B_1'}\frac{\na_{x'}f_\alpha^k\cdot n'}{\sqrt{1+|\na_{x'}f_\alpha^k|^2}}-
\delta_k\H^{n-2}(B_1') + \delta_k \H^{n-1}(\fb u\cap Q).
\end{eqnarray*}
Now the result follows from the lower semicontinuity of the first variations  \cite{Simon} Theorem 40.6. 
\subsection{Allard's monotonicity formula}
We recall Theorem 17.6 from \cite{Simon}:
Let $U\subset \R^n$ be a bounded domain and $\xi\in U, 0<\alpha\le 1$, $\Lambda\ge 0$
such that $V$ has bounded first variation, the mean curvature $H$ is a $\mu_V$ integrable function
satisfying  
\begin{equation*}
\frac1{\alpha}\int_{B_\rho(\xi)}|H|\le \lambda\left(\frac\rho R\right)^{\alpha-1}\mu(B_\rho(\xi))\quad \forall \rho\in(0, R), 
\end{equation*}
where $\overline{B_R(\xi)}\subset U$, then $f(\rho)=e^{\Lambda R^{1-\alpha}\rho^\alpha}\frac{\mu(B_\rho(\xi))}{\rho^{n-1}}$ is non decreasing function of $\rho$, and in fact 
\begin{equation*}
f(\sigma)\le f(\rho)-\int_{B_\rho(\xi)\setminus B_{\sigma}(\xi)}\frac{|D^\perp r|^2}{r^{n-1}}.
\end{equation*}
Furthermore, we have the following 
\begin{thm}\label{thm:Allard}
If 
\begin{equation}\label{eq:H-dens}
\int_{B_\rho(\xi)}|H|\le  C(\xi, \alpha) r^{n-2+\alpha}, \quad r\in (0, R), \overline{B_R}(\xi)\subset U
\end{equation} 
then 
$\Gamma$ has a tangent cone at $\xi$. 
\end{thm}

The assumption \eqref{eq:H-dens} implies that the density of $V$ at $\xi$ exists, in view of Theorem 17.7
and Corollary 17.8 \cite{Simon}. Then the existence of tangent cone follows from Theorem 19.3 \cite{Simon}.

\section{Second variation formula}\label{sec:4}
The second variation of the energy for the classical case $f(t)=t$ 
has been computed in 
\cite{CJK}. 
Our computation is more involved due to the nonlinear form of $F$. 
The main result of this section is 
\begin{thm}
Let $u\ge 0$ be a local minimizer of $J_F$ in $B_1$, and $0\in \fb u$ such that 
$\fb u\setminus \{0\}$ is smooth. Then for every $\psi\in C_0^\infty(B_1\setminus\{0\})$ there
holds 
\begin{eqnarray}\label{eq;2nd-var-ineq}
\int_\Gamma H \psi^2
 \le 
 \frac1{F'(1)} \int_{\Om^+(u)} F'(|\na u|^2)\left\{|\na \psi|^2+\frac{2F''(|\na u|^2)}{F'(|\na u|^2)}(\na u\na \psi)^2\right\},
\end{eqnarray}
where $H$ is the mean curvature of $\fb u$. 
\end{thm}

\begin{remark}\label{rem:Bernoulli}
Recall that the free boundary condition has the form 
\begin{equation}\label{fb-cond-0}
f(\na u)+\lambda^2-\na_\xi f(\na u)\na u=0, \quad f(\xi)=F(|\xi|^2).
\end{equation}
This condition gives that $|\na u|=\lambda^*$ on the free boundary, where 
$\lambda^*$ is determined from $\lambda^2=2F'(|\na u|^2)|\na u|^2-F(|\na u|^2)$. We normalize  the constant $\lambda^*$ such that 
\begin{equation}
|\na u|=1 \ \ \text{on}\ \ \Gamma.
\end{equation}
To do so we take  $v=cu$ for a suitable constant $c$ and consequently we see that  that $|\na v|=1$ on $\Gamma$. 
\end{remark}

\begin{proof}
Take $\e>0$  and $0\le \psi\in C_0^{\infty}(B_1\setminus\{0\})$ and introduce 
\begin{equation}
u_\e=\max(u-\e\psi, 0)=\left\{
\begin{array}{lll}
u-\e\psi & \text{if} \ u>\e\psi, \\
0 & \text{otherwise}.
\end{array}
\right.
\end{equation}
We have 
\begin{eqnarray}\label{f-energy}
\int _{\Omega _{R}^+\left( u_{\varepsilon }\right) }f\left( \nabla u_{\varepsilon }\right) 
&=&
\int _{\left\{ u > \varepsilon \psi\right\} \cap B_R}f\left( \nabla (u-{\varepsilon }\psi)\right)\\\nonumber
&=&
\int _{\Omega _{R}^+\left( u\right) }f\left( \nabla (u-{\varepsilon }\psi)\right)-\int_{\{0<u<\e \psi\}\cap B_R} f\left( \nabla (u-{\varepsilon }\psi)\right)\\\nonumber
&=&
I_1-I_2, 
\end{eqnarray}
where 
\[
I_1=\int _{\Omega _{R}^+\left( u\right) }f\left( \nabla (u-{\varepsilon }\psi)\right),
\]
and 
\[
I_2=\int_{\{0<u<\e \psi\}\cap B_R} f\left( \nabla (u-{\varepsilon }\psi)\right).
\]
We first simplify $I_1$. 
Let us expand $f\left( \nabla u-{\varepsilon }\na\psi\right) $ in $\e$ in order to get 
\begin{equation}\label{eq:blya-1345}
f\left( \nabla u-{\varepsilon }\na \psi\right) =A_0+\e A_1+\e^2A_2+O(\e^3).
\end{equation}
In fact the coefficients $A_0, A_1, A_2$ can be computed explicitly. 
Precisely, if $f(\na u)=F(|\na u|^2)$ then  
it follows 
\begin{eqnarray*}
f\left( \nabla u-\varepsilon \nabla \psi \right) 
&=&
F\left( \left| \nabla u\right| ^{2}-2\varepsilon \nabla u\nabla \psi +\varepsilon ^{2}\left| \nabla \psi \right| ^{2}\right) \\
&=&
F\left(|\na u|^2+\e(\e|\na \psi|^2-2\na u\na u)\right)\\
&=& F(|\na u|^2)+F'(|\na u|^2)\e(\e|\na \psi|^2-2\na u\na u)\\
&&
+\frac{F''(|\na u|^2)}2\e^2(\e|\na \psi|^2-2\na u\na u)^2+ O(\e^3)\\
&=&
F(|\na u|^2)-2F'(|\na u|^2)\na u\na \psi \e\\
&&+\left[F'(|\na u|^2)|\na \psi|^2+2F''(|\na u|^2)(\na u\na \psi)^2\right]\e^2 +O(\e^3)\\
&=&
A_0+A_1\e+A_2\e +O(\e^3)
\end{eqnarray*}
where 
\begin{eqnarray}\label{def:A0}
A_0
&=&
F(|\na u|^2), \\ \label{def:A1}
A_1
&=&
-2F'(|\na u|^2)\na u\na \psi, \\ \label{def:A2}
A_2
&=&
F'(|\na u|^2)|\na \psi|^2+2F''(|\na u|^2)(\na u\na \psi)^2.
\end{eqnarray}

Consequently 
\begin{eqnarray}\label{Taylor-1}
I_1=\int _{\Omega _{R}^+\left( u\right) }f\left( \nabla (u-{\varepsilon }\psi)\right)
=
\int _{\Omega _{R}^+\left( u\right) } \left[A_{0}+A_{1}\varepsilon +A_{2}\varepsilon ^{2}+O\left( \varepsilon ^{3}\right) \right]
\end{eqnarray}

\medskip

Next, we simplify  $I_2$. We  assume that $\fb u$ is parametrized by $z(s), s\in D$ for some domain $D\subset \R^{n-1}$ then we can express the points of $\{0<u<\e\psi\}$ as
$x(s, t)=z(s)-\nu(z(s))t$ where $\nu(z)$ is the unit outer normal of $\{u>0\}$ at $z$ and $0<t<t_\e(z(s))$.   
Using \eqref{eq:blya-1345} we obtain

\begin{eqnarray}\nonumber
\int _{\left\{ 0 < u < \varepsilon \psi \right\} \cap B_{R}}f\left( \nabla u-\varepsilon \nabla \psi \right) 
&=&
\int _{\left\{ 0 < u < \varepsilon \psi \right\} \cap B_{R}} \left(\sum_{i=0}^2\e^i A_i(x)+O(\e^3)\right)dx
\\ \label{blya-11}
&=&
\int_D \int^{t_\e(z)}_0 \left(\sum_{i=0}^2\e^i A_i(z+t_\e(y)\na u(z))+O(\e^3)\right)\|\det J(s, t)\|dtds,
\end{eqnarray}
where $\mathcal J(z, t)=\frac{\p(x)}{\p(s, t)}$ is the transformation matix. Note that at $s_0\in D$ we can rotate the coordinate system such that at $z(s_0)$ we have $\na u(z(s_0))=-e_n$. Then it follows that at $s_0$ we have 

\begin{eqnarray}\nonumber
\|\mathcal J(s_0, t)\|
&=&
\det
\begin{Vmatrix}
    z_{s_1}^1+tu_{1m}z^m_{s_1} & z_{s_2}^1+tu_{1m}z^m_{s_2}  & \dots  & 0 \\
z_{s_1}^2+tu_{2m}z^m_{s_1} & z_{s_2}^2+tu_{2m}z^m_{s_2}  & \dots  & 0 \\
    \vdots &   \vdots & \ddots &\vdots \\
z_{s_1}^n+tu_{nm}z^m_{s_1} & z_{s_2}^n+tu_{nm}z^m_{s_2}  & \dots  & -1 \\
\end{Vmatrix}\\\nonumber
&=&
\| z^i_{s_j}\|\|(\delta_{ij}+tD^2_{s_is_j}u)\| \quad 1\le i, j\le n-1\\
&=&
 \| z^i_{s_j}\|(1+t\sum_{i=1}^{n-1}u_{ii}+O(t^2)).\label{vol-term}
\end{eqnarray}

Recall that \cite{GT} section 14.6 the mean curvature is 
\begin{eqnarray}\label{mean-curv}
H(z(s_0))
&=&
\frac1{|\na u (z(s_0))|}\left(\Delta u(z(s_0)) -\frac{\na u(z(s_0)) D^2 u(z(s_0))\na u(z(s_0))}{|\na u(z(s_0))|^2}\right)\\\nonumber
&=&
\sum_{i=1}^{n-1}u_{ii}(z(s_0).
\end{eqnarray}

Returning to \eqref{blya-11} and noting that $\|z^i_{s_j}\|ds=d\H^{n-1}(z), 1\le i, j\le n-1$
we get 

\begin{eqnarray}\nonumber
\int _{\left\{ 0 < u < \varepsilon \psi \right\} \cap B_{R}}f\left( \nabla u-\varepsilon \nabla \psi \right) 
=
\int_\Gamma \int^{t_\e(z)}_0 \left(\sum_{i=0}^2\e^i A_i(z+t\na u(z))+O(\e^3)\right) (1+t\sum_{i=1}^{n-1}u_{ii}+O(t^2))dtd\H^{n-1}.
\end{eqnarray}

From Taylor's formula we have 
\begin{eqnarray*}
A_{i}(z+t \na u(z))
=
A_{i}\left( z \right) 
+
t \nabla A_{i}\left( z \right) \nabla u\left( z\right) 
+
\dfrac {t^{2}}{2}(\na^2 A_i(z)\na u(z))\na u(z)+O(t^3).
\end{eqnarray*}
Let $H$ be the mean curvature of $\Gamma$ and introduce 
$$I_{2i}=\int _{\left\{ 0 < u < \varepsilon \psi \right\} \cap B_{R}} A_i, \quad i=0, 1,2,$$
so that 
\[
I_2=\sum_{i=0}^2\e^{i+1}I_{2i}.
\]
We have from above computations  
\begin{eqnarray*}
I_{2i}
&=&
\int_\Gamma\int_0^{t_\e}
\left\{
A_{i}\left( z\right) 
+
t\nabla A_{i}\left( z \right) \nabla u\left( z\right) 
+
\dfrac {t^{2}}{2}(\na^2 A_i(z)\na u(z))\na u(z)
+
O(t^3)
\right\}
\left[1+Ht+O(t^2)\right]dtd\H^{n-1}\\
&=&
\int_\Gamma A_i\left(t_\e+\frac{H}2 t_\e^2+O(t_\e^3)\right)
+
\int_\Gamma \na A_i \na u \left( \frac{t_\e^2}2+ H\frac{t_\e^3}3+O(t_\e^4) \right).
\end{eqnarray*}

Note that $u(z+t_\e\na u(z))=\e\psi(z+t_\e\na u(z))$ on $\p\{u>\e\psi\}$, hence taking Taylor's expansion
we obtain 
\begin{equation}\label{t-expand}
t_\e=\e\psi+\e^2\left(-\psi^2\frac{u_{\nu\nu}}2-\psi\psi_\nu\right)+O(\e^3).
\end{equation}
From here 

\begin{eqnarray*}
I_{2i}
&=& 
\int_\Gamma A_i\e\psi+\int_\Gamma A_i\left[\frac{H}2\e^2\psi^2+\e^2\left(-\psi^2\frac{u_{\nu\nu}}2-\psi\psi_\nu\right)\right]\\
&&
+\int_\Gamma \na A_i\na u\frac{\e^2\psi^2}2 +O(\e^3).
\end{eqnarray*}

\medskip

Denoting $\widetilde H=\frac12(H-u_{\nu\nu})$ we further  simplify 
\begin{eqnarray}\nonumber
I_2&=&
\sum_{i=0}^2 \e^{i+1}\int_\Gamma \left\{A_i\psi+\e A_i\left[{\widetilde H}\psi^2-\psi\psi_\nu\right]+\e\na A_i\na u\frac{\psi^2}2\right\}+O(\e^3) \\\nonumber
&=&
\e\int_\Gamma \left\{A_0\psi+\e A_0\left[{\widetilde H}\psi^2-\psi\psi_\nu\right]+\e\na A_0\na u\frac{\psi^2}2\right\}\\\nonumber
&&
- \e^2\int_\Gamma \left\{A_1\psi+\e A_1\left[{\widetilde H}\psi^2-\psi\psi_\nu\right]+\e\na A_1\na u\frac{\psi^2}2\right\}+O(\e^3)\\
&=&\label{blya-12}
\e\int_\Gamma A_0\psi-\e^2\int_\Gamma A_1\psi+A_0\left[{\widetilde H}\psi^2-\psi\psi_\nu\right]+
\na A_0\na u\frac{\psi^2}2+O(\e^3).
\end{eqnarray}

Plugging \eqref{Taylor-1} and \eqref{blya-12} into \eqref{f-energy}
we get the formula
\begin{eqnarray}\label{I-sum}
\int_{\Omega^+_R(u_\e)}f(\na u_\e)
&=&
I_1-I_2.
\end{eqnarray}

From \eqref{vol-term}, \eqref{mean-curv} and \eqref{t-expand} we get 
\begin{eqnarray}\nonumber
|\{0<u<\e\psi\}\cap B_R|
&=&
\int_\Gamma\int_0^{t_\e}(1+Ht+O(t^2))dtd\H^{n-1}=\int_\Gamma(t_\e+H\frac{t^2_\e}2+O(t_\e^3))d\H^{n-1}\\
&=&\nonumber
\int_\Gamma(\e\psi+\e^2\left(-\psi^2\frac{u_{\nu\nu}}2-\psi\psi_\nu\right)+\e^2\frac H2 \psi^2) +O(\e^3)\\
&=&\label{small-vol}
\int_\Gamma\e\psi+\e^2(\widetilde H\psi^2-\psi\psi_\nu) +
O(\e^3).
\end{eqnarray}

Now the comparison of energies yields

\begin{eqnarray}\nonumber
0\ge J_F(u)-J_F(u_\e)
&\stackrel{\eqref{Taylor-1}}{=}&\nonumber
\int_{B_R} F(|\na u|^2)+\lambda^2\I {u}-\int_{\Omega^+_R(u_\e)}F(|\na u_\e|^2)-\lambda^2\int_{B_R}\I{u-\e\psi}\\\nonumber
&\stackrel{\eqref{I-sum}}{=}&\label{2nd-var-main}
\int_{B_R} F(|\na u|^2)-I_1\\ \nonumber
&&+I_2+\lambda^2\int_{B_R}\chi_{\{0<u<\e\psi\}}\\ 
&\stackrel{\eqref{small-vol}}{=}&
-\e\int_{\Om^+_R(u)} A_1-\e^2\int_{\Om^+_R(u)} A_2+\\ \nonumber
&& +\e\int_\Gamma A_0\psi+\e^2\int_\Gamma A_1\psi+A_0(\widetilde H\psi^2-\psi\psi_\nu)+\na A_0\na u\frac{\psi^2}2\\ \nonumber
&&
+\lambda^2\int_\Gamma\e\psi+\e^2(\widetilde H\psi^2-\psi\psi_\nu) +
O(\e^3).
\end{eqnarray}

Recall the free boundary condition  \eqref{fb-cond-0}
\begin{equation*}
\lambda^2+ f(\na u)-\na_\xi f\na u=0, 
\end{equation*}
hence it follows from the divergence theorem and \eqref{def:A0}-\eqref{def:A1} that the coefficient of $\e$ in the expression above vanishes. 
On the other hand we see that the coefficient of $\e^2$ in \eqref{2nd-var-main} is

\begin{eqnarray*}
-\int_{\Om^+_R(u)} A_2
+
\int_\Gamma A_1\psi
+
A_0(\widetilde H\psi^2-\psi\psi_\nu)+\na A_0\na u\frac{\psi^2}2
+
\lambda^2\int_\Gamma(\widetilde H\psi^2-\psi\psi_\nu).
\end{eqnarray*}
Thus letting $\e\to 0$ we obtain the inequality

\begin{eqnarray}\label{blya-ineq-1}
\int_\Gamma A_1\psi
+
A_0(\widetilde H\psi^2-\psi\psi_\nu)+\na A_0\na u\frac{\psi^2}2
+
\lambda^2\int_\Gamma(\widetilde H\psi^2-\psi\psi_\nu)\le \int_{\Om^+_R(u)} A_2.
\end{eqnarray}

From \eqref{def:A2} we have that 
\begin{eqnarray}\label{blya-A2}
\int_{\Om^+_R(u)} A_2
&=&
 \int_{\Om^+_R(u)} F'(|\na u|^2)|\na \psi|^2+2F''(|\na u|^2)(\na u\na \psi)^2\\\nonumber
 &=&
  \int_{\Om^+_R(u)} F'(|\na u|^2)\left\{|\na \psi|^2+\frac{2F''(|\na u|^2)}{F'(|\na u|^2)}(\na u\na \psi)^2\right\}.
\end{eqnarray}
%
Denoting 
\[
I_3=\int_\Gamma A_1\psi
+
A_0(\widetilde H\psi^2-\psi\psi_\nu)+\na A_0\na u\frac{\psi^2}2
+
\lambda^2\int_\Gamma(\widetilde H\psi^2-\psi\psi_\nu)
\]
and recalling \eqref{def:A0},\eqref{def:A1} we get 
\begin{eqnarray*}
I_3
&\stackrel{\eqref{def:A0},\eqref{def:A1}}{=}&
\int_\Gamma  -2F'\na u\na \psi\psi-F\psi\psi_\nu-\lambda^2\psi\psi_\nu\\
&&+
\int_\Gamma \left[F\widetilde H+ \lambda^2\widetilde H+F'\na uD^2u\na u\right]\psi^2\\
&=&
\int_\Gamma  (2F' -F-\lambda^2)\psi_\nu\psi\\
&&+
\int_\Gamma \left[F\widetilde H+ \lambda^2\widetilde H+F'\na uD^2 u \na u\right]\psi^2\\
&\stackrel{\eqref{fb-cond-0}, \eqref{mean-curv}}{=}&
\int_\Gamma \left[F\widetilde H+ \lambda^2\widetilde H+F'u_{\nu\nu}\right]\psi^2\\
&\stackrel{\eqref{fb-cond-0}}{=}&
\int_\Gamma \left[2F'\widetilde H+F'u_{\nu\nu}\right]\psi^2\\
&=&
\int_\Gamma \left[F'(H-u_{\nu\nu})+F'u_{\nu\nu}\right]\psi^2\\
&=&
\int_\Gamma F'H\psi^2\\
&=&F'(1)\int_\Gamma H\psi^2.
\end{eqnarray*}

Combining this with \eqref{blya-A2} and \eqref{blya-ineq-1} we obtain \eqref{eq;2nd-var-ineq}.
\end{proof}

\section{Stability operator }\label{sec:5}

Let $a_{ij}$ be the matrix 
\begin{equation}
a_{ij}= \frac{F'(|\na u|^2)}{F'(1)} \left\{\delta_{ij}+\frac{2F''(|\na u|^2)}{F'(|\na u|^2)}u_iu_j\right\}
\end{equation}
and define the operator 
\begin{equation}
L[\psi]=\div(a\na \psi).
\end{equation}
Then \eqref{eq;2nd-var-ineq} tells us that 
\begin{equation}
\int_{\Om^+(u)}(L[\psi]+H\psi)\psi\le 0, \quad \forall \psi\in C_0^\infty(B_1\setminus\{0\}).
\end{equation}
Thus $L[\psi]+H\psi$ is the stability operator for the minimization problem.

We can make our computation more explicit.
For instance if $f(\xi)=|\xi|^p, 2<p<\infty$ then 
\begin{eqnarray*}
\left| \nabla \left( u-\varepsilon \psi \right) \right| ^{p}
&=&
\left( \left| \nabla u\right| ^{2}-2\varepsilon \nabla u\nabla \psi +\varepsilon ^{2}\left| \nabla \psi \right| ^{2}\right) ^{p/2}\\
&=&
\left( \left| \nabla u\right| ^{2}+\varepsilon \left( \varepsilon \left| \nabla \psi \right|^2-2\nabla u\nabla \psi \right) \right) ^{p/2}\\
&=&
\left| \nabla u\right| ^{p}+\dfrac {p}{2}\left| \nabla u\right| ^{p-2} \varepsilon \left( \varepsilon \left| \nabla \psi \right| ^{2}-2\nabla u\nabla \psi \right) \\
&&
+\frac12 \dfrac {p}{2}\left( \dfrac {p}{2}-1\right) \left| \nabla u\right| ^{p-4}\varepsilon ^{2}\left[ \varepsilon \left| \nabla \psi \right| ^{2}-2\nabla u\nabla \psi \right] ^{2}+O\left( \varepsilon ^{3}\right)\\
&=& 
\left| \nabla u\right| ^{p}-\varepsilon p\left| \nabla u\right| ^{p-2}\nabla u\nabla \psi 
+\varepsilon ^{2}
\left[
\dfrac {p}{2}\left| \nabla u\right| ^{p-2}\left| \nabla \psi \right| ^{2}+4\dfrac {p}{4}\left( \dfrac {p}{2}-1\right) \left( \nabla u\nabla \psi \right) ^{2}\left| \nabla u\right| ^{p-4}
\right]+
O\left( \varepsilon ^{3}\right) \\
&=& 
\left| \nabla u\right| ^{p}-\varepsilon p\left| \nabla u\right| ^{p-2}\nabla u\nabla \psi +
\varepsilon ^{2}
\left[
\dfrac {p}{2}\left| \nabla u\right| ^{p-2}\left| \nabla \psi \right| ^{2}+\frac{p}2\left( p-2\right) \left( \nabla u\nabla \psi \right) ^{2}\left| \nabla u\right| ^{p-4}
\right]
+
O\left( \varepsilon ^{3}\right) \\
&=&
A_{0}+A_{1}\varepsilon +A_{2}\varepsilon ^{2}+O\left( \varepsilon ^{3}\right), 
\end{eqnarray*}
where 
\begin{eqnarray*}
A_0&=&\left| \nabla u\right| ^{p},\\
A_1&=&-p\left| \nabla u\right| ^{p-2}\nabla u\nabla \psi,\\
A_2&=&
\dfrac {p}{2}\left| \nabla u\right| ^{p-2}\left| \nabla \psi \right| ^{2}+\frac{p}2\left( p-2\right) \left( \nabla u\nabla \psi \right) ^{2}\left| \nabla u\right| ^{p-4}.\\
%
\end{eqnarray*}
Splitting the integral as in \eqref{f-energy}
we get that 
\begin{eqnarray*}
\int _{\Omega _{R}^+\left( u_{\varepsilon }\right) }\left| \nabla u_{\varepsilon }\right| ^{p}
&=&
\int _{\left\{ u > \varepsilon \psi\right\} \cap B_R}\left| \nabla \left( u-\varepsilon \psi\right) \right| ^{p}\\
&=&
\int _{\Omega _{R}^+\left( u\right) }\left| \nabla \left( u-\varepsilon \psi\right) \right| ^{p}-\int_{\{0<u<\e \psi\}\cap B_R} |\nabla \left( u-\varepsilon \psi\right)|^p 
\end{eqnarray*} 
and we can carry over the computation of previous section. 
In this case \eqref{eq;2nd-var-ineq} takes form 
\begin{eqnarray}\label{eq:xxxx}
\int_\Gamma H \psi^2
 \le 
 \int_{\Om^+(u)} |\na u|^{p-2}\left\{|\na \psi|^2+(p-2)\frac{(\na u\na \psi)^2}{|\na u|^2}\right\}.
\end{eqnarray}

\section{Proof of main theorem}\label{sec:6}
First we prove part (a) of Main Theorem.
\begin{thm}\label{thm:Bernstein}
Let $u\ge 0$ be a homogeneous global solution then 
$\fb {u}$ consists of finitely many convex cones.
Moreover, if $\sup{F''(|\na u|^2)}<\frac32 {F'(1)}$ then $u$ is a half-plane solution.   
\end{thm}
\begin{proof}

First we estimate the geodesic curvature of $\fb u$.
Let us choose 
$g, h\in C^\infty[0, \infty)$ $g, h \ge 0$
where 
$$
g(r)=
\begin{cases}
1 \quad &\text{if} \ 0\le r\le \frac12,\\
0 &\text{if} \  r\ge\frac34,
\end{cases}
$$
and 
\[
h(r)=
\begin{cases}
1 \quad &\text{if} \ 0\le r\le 1,\\
0  &\text{if} \ r\ge2. 
\end{cases}
\]
Let us take  $\psi_\e(r)=h(r/\e)g(r)r^{-1/2}$,  where $\e>0$ is small, and compute the integrals in the 
inequality \eqref{eq;2nd-var-ineq} for this choice of the test function. 
We have 
\[
\int ^{}_{\gamma }\int ^{\infty }_{0}\kappa\left( \alpha \left( s\right) \right) \left[\psi_\e\left( r\right) \right]^{2}drds
\leq 
\widehat C \int ^{+\infty }_{0}\left[ \psi_\e'\left( r\right) \right] ^{2}r^{2}dr\H^2\left(\po u\cap \S^2\right)
\]
where $\gamma=\fb u \cap \S^2$, $\widehat C=1+2\sup\frac{F''(|\na u_0|^2)}{F'(1)}$, and 
$\kappa$ is the geodesic curvature of $\gamma$. Observe that 

\begin{eqnarray*}
\int ^{+\infty }_{0}\left[ \psi_{\varepsilon }'\left( r\right) \right] ^{2}r^{2}dr
&=&
\int ^{+\infty }_{0}\left[ -\dfrac {1}{2}r^{-3/2}h\left( r/\varepsilon \right) g\left( r\right) \right] ^{2}r^{2}dr+O\left( 1\right) \\
&=&
\frac14\int ^{+\infty }_{0}\left[ r^{-1/2}h\left( r/\varepsilon \right) g\left( r\right) \right] ^{2}dr+O\left( 1\right) \\
&=&
\frac14\int ^{+\infty }_{0}\left[\psi_\e(r) \right] ^{2}dr+O\left( 1\right).  \\
\end{eqnarray*}

On the other hand  
\begin{equation*}
\int ^{+\infty }_{0}\left[\psi_\e(r) \right] ^{2}dr=c\log 1/\e +O\left( 1\right) \\
\end{equation*}
for some $c>0$ hence 

\begin{equation*}
\lim_{\e\to 0}\frac{\int ^{+\infty }_{0}\left[ \psi_{\varepsilon }'\left( r\right) \right] ^{2}r^{2}dr}{\int ^{+\infty }_{0}\left[\psi_\e(r) \right] ^{2}dr}=\frac14.
\end{equation*}

Consequently we get 
\begin{equation*}
\int kds\le \frac{\widehat C}4 \H^2\left(\po u\cap \S^2\right).
\end{equation*}

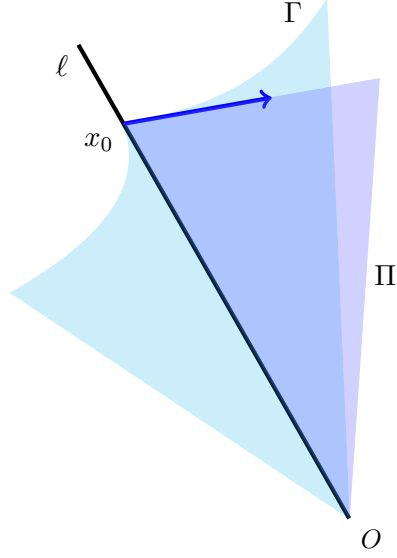
\begin{figure}
\begin{center}
 \begin{tikzpicture}

    \coordinate (O) at (3,-5.25) ;
    \coordinate (A) at (0,0) ; 
        \coordinate (A1) at (-0.6,+1.05) ; 
    \coordinate (B) at (2.7,1.65) ; 
        \coordinate (C) at (-1.5,-2.25) ; 

 \draw[ultra thick] (O) node[below right] {$\resizebox{0.017\hsize}{!}{$O$}$}--(A1) ;  
         \node[below left] at (A)  {{$x_0$}};
                  \node[below left] at (A1)  {$\resizebox{0.01\hsize}{!}{$\ell$}$};

        \filldraw[cyan,opacity=.2] (A) to[out=10,in=237] (B)--(O); 
        \filldraw[cyan,opacity=.2] (A) to[out=-70,in=30] (C)--(O); 

              \draw[->, blue, ultra thick ] (A)-- (10:2) ;
      \filldraw[blue!60!,opacity=.3] (A)--(3.4,0.6)--(O);

                \node[below right] at (2,1.75) {{$\Gamma$}};
                                \node[below right] at (3.2,-1.75) {{$\Pi$}};

\end{tikzpicture}
\end{center}
\caption{Possible singularity at $x_0$. }
\label{fig-3}
\end{figure}

We observe that the free boundary is smooth away from the vertex. Indeed, from \eqref{eq:mincurv-H} we know that the free boundary components are convex cones and suppose 
there is a ray $\ell$ such that $\Gamma$ has singularity along $\ell$, see Figure \ref{fig-3}.
If we blow-up $u$ at some $x_0\in \ell\setminus \{0\}$ then the free boundary of the blow-up limit (which exists thanks to the compactness of blow-up sequences $u_j(x_0+r_j)/r_j, r_j\to 0$ see section 3 \cite{ACF-quasi})
will contain a flat portion $\Pi.$ Then as in the proof of Theorem \ref{thm:mean-curv} we infer that 
the blow-up must have smooth free boundary which is a contradiction.

Now we apply the {Gauss-Bonnet theorem}, \cite{ON} page 375, to infer 
\[
\H^2(V_i)+\int_{\gamma_i}\kappa=2\pi
\]
where $V_i$ is the component of $\S^2\cap \{u=0\}$ and $\kappa$ is the 
curvature of $\gamma_i=\p V_i$. From here we obtain 
\[
4\pi -\H^2(\S^2\cap \Om^+)=2\pi m-\int_{\Gamma\cap \S^2}\kappa
\]
or 

\[
2\pi(m-2)+\H^2(\S^2\cap \Om^+)=\int_{\Gamma\cap \S^2}\kappa\le  \frac{\widehat C}4 \H^2\left(\po u\cap \S^2\right),
\]
which yields the estimate 
\[
2\pi(m-2)\le\left(\frac{\widehat C}4 -1\right)\H^2\left(\po u\cap \S^2\right).
\]
Clearly if $\widehat C<4$ then $m=0, 1$.

Consequently the free boundary is Lipschitz graph near $0$. Since every minimizer is also 
a viscosity solution \cite{DK} then applying the $C^{1, \alpha}$ regularity result for Lipschitz graphs 
\cite{Feldman} we conclude that $\Gamma$ cannot have singularity at $0$.
\end{proof}

\begin{remark}
 For $F(t)=t^{\frac p2}$ we have that 
$\widehat C={p-1}$ hence for 
\[
2<p<5, \quad n=3
\]
the homogeneous free boundary is flat. 
\end{remark}

\subsection{Partial regularity}
Now we prove part (b) of Main Theorem. 
\begin{thm}\label{thm:partial}
Suppose $V=v(\Gamma, \theta)$ has bounded first variation. If $n=3$ and  \eqref{small-ellipt} is satisfied
then $\Gamma$ is smooth away from a  set $\Sigma_0$ of singular points and $\dim(\Sigma_0)\le 1$.
\end{thm}
\begin{proof}
Let $\alpha\in(0, 1)$ and $\e^\ast>0$ be fixed. Let $E(\alpha, \e^\ast)$ be the set of the points $x\in \Gamma$ where 
\begin{equation}\label{blya-diuvhfiuhr}
\frac1{r^{n-1}}\int_{B_r(x)}Hd\mu_V\ge \e^\ast r^{\alpha-1} \quad \mbox{whenever}\ r<r_x
\end{equation}
 for some small $r_x>0$  depending on $x$. Applying Theorem \ref{thm:Allard} we see
 that $\Gamma$ at $x$ has tangent cone. Then Theorem \ref{thm:Bernstein} implies that 
 that $\Gamma$ is smooth at $x$ if $n=3$. Hence $E(\alpha, \e^\ast)\subset \fbr u$ if $n=3$. 
 
 Suppose that $y\in \Gamma\setminus F(\alpha, \e^\ast)$. Then we can choose balls $B_i$ with following properties;  $\Gamma\setminus F(\alpha, \e^\ast)\subset \cup_i B_i$, $\diam B_i\le \delta$ for some small $\delta>0$ and
 \[
 \int_{B_i\cap \Gamma}Hd\mu_V\ge \e^\ast (\diam B_i)^{n-2+\alpha}.
 \]
 Then using Vitali's covering theorem Theorem 3.3 \cite{Simon} we see that the $\delta$ Hausdorff premeasure satisfies the estimate
 \[
 \H^{1+\alpha}_\delta((\Gamma\cap U)\setminus E(\alpha, \e^\ast))\le \int_{(\Gamma\cap U)\setminus E(\alpha, \e^\ast)}Hd\mu_V.
 \]
 Sending $\delta\to 0$ and recalling that 
 $\alpha<1$ we conclude that $\H^{1+\alpha}((\Gamma\cap U)\setminus E(\alpha, \e^\ast))=0$.
Since the singular points must be in the complement of $E(\alpha, \e^\ast)$ then it follows that 
$\dim(\Sigma)\le 1$. 
\end{proof}

\subsection{The smooth free boundaries in $\R^2$ are lines}
Our last result is a simple consequence from the stability inequality 
\eqref{eq;2nd-var-ineq} in $\R^2$ where one can use the trick of logarithmic test function.  
\begin{thm}
Suppose that $u$ is a global minimizer of the energy \eqref{eq:J-quasi} in $\R^2$ such that the free boundary of $u$ is smooth.
Then $u=x_1^+$ after some rotation of the coordinate system.
\end{thm}

\begin{proof}
By compactness \eqref{eq;2nd-var-ineq} is valid for Lipschitz $\psi$.
Choose
$$
\psi(r)=
\begin{cases}
 1  & \text{if}\  r\le e^N, \\
2-\frac{\log r}N & \text{if}\  e^N<r\le e^{2N},\\
0  &\text{if}\  r>e^{2N}.
\end{cases}
$$
Therefore we get from \eqref{eq;2nd-var-ineq}
\begin{eqnarray*}
\int_{B_{e^N}\cap \fb u}H
&\le&
\widehat C\int_{B_{e^{2N}}}|\na \psi |^2\\
&=&
2\pi \widehat C\int_{e^N}^{e^{2N}}\frac1{r^2N^2}rdr=\frac{2\pi \widehat C}N\to 0\ \ \text{as}\ \ N\to \infty.
\end{eqnarray*}
Since $H=\kappa\ge 0$ it follows that $\fb u$ is a line $\ell$.
After odd reflection of $u$ across $\ell$ we can apply Liouville's theorem to conclude that 
the reflected function is linear and hence the desired result follows. 
\end{proof}

\section{The double cone solution to  $p$-laplacian in $\R^3$}\label{sec:7}
In this section we consider the case $F(t)=t^{p/2}, 1<p<\infty$. The full regularity 
of minimizers in $\R^2$ is proved in \cite{DP}.  

\subsection{$p$-Legendre equations}
For $n=3$ and $x=\rho (\cos\phi\sin \theta, \sin\phi\sin\theta, \cos\theta )$ we put
 $$u=\displaystyle\rho \max\left(\frac{f(\theta)}{\dot f(\theta_0)}, 0\right), $$
where $f$ is the solution of the differential equation 
\begin{eqnarray}\label{Legendre}
 \ddot f +(f+\cot \theta \dot f)\frac{f^2+\dot f^2}{f^2+(p-1)\dot f^2}+f=0,
\end{eqnarray}
and $\theta_0$ is the only zero of $f$ in $(0, \pi/2]$, cf. \cite{AC} 2.7.

It is worth to point out that for $p=2$ (\ref{Legendre}) is the Legendre equation for $n=1$, \cite{Kellogg} page 127,
and in this case $Q_1(x)=y(x)=1-\frac x2\ln\frac{1-x}{1+x}$.

Rewriting the divergence of a vectorfield $W=(W_1, W_2, W_3)$ 
in spherical coordinates 
\cite{Kellogg} page 183 (12), we get
$$\div W=\frac{1}{\rho^2\sin\theta}\left\{\frac{\partial }{\partial \rho}[\rho^2\sin\theta W_1]
+ \frac{\partial }{\partial \phi}[\rho W_2]
+ \frac{\partial }{\partial \theta} [\rho\sin\theta W_3]\right\}.$$

Moreover,  $\na u=(u_\rho, \frac 1{\rho\sin\theta}u_\phi, \frac 1\rho u_\theta)$,  \cite{Kellogg} page 181 (10), hence 
for $W=|\na u|^{p-2}\na u$ the above formula for divergence yields
\begin{eqnarray*}
 \Delta_p u&=&\div W\\\nonumber
&=& \frac{1}{\rho^2\sin\theta}\left\{\frac{\partial }{\partial \rho}[\rho^2\sin\theta |\na u|^{p-2}u_\rho]
+ \frac{\partial }{\partial \phi}\left[\frac 1{\sin\theta} |\na u|^{p-2}u_\phi\right]
+ \frac{\partial }{\partial \theta} \left[\sin\theta |\na u|^{p-2}u_\theta\right]\right\}.
\end{eqnarray*}
Since $u$ does not depend on $\phi$ and is homogeneous function of degree 1 it follows that
$$
 2 \sin \theta \left[f^2+\dot f^2\right]^{\frac{p-2}2}f+\frac d{d\theta}\left\{\sin \theta \left[f^2+\dot f^2\right]^{\frac{p-2}2}\dot f\right\}=0,
$$
or equivalently 

\[
\ddot f(f^2+(p-1)\dot f^2)+f(2f^2+p \dot f^2)+\cot\theta\dot f (f^2+\dot f^2)=0 
\]
which is (\ref{Legendre}). Thus $\Delta_pu=0$ in $\{u>0\}$ and moreover
$|\na u|=\lambda^\ast$ on $\partial \{u>0\}\setminus \{ 0\}$ 
with $(p-1)^{-\frac1p}=\lambda^\ast$. The free boundary of $u$ is a double cone, see Figure 2.

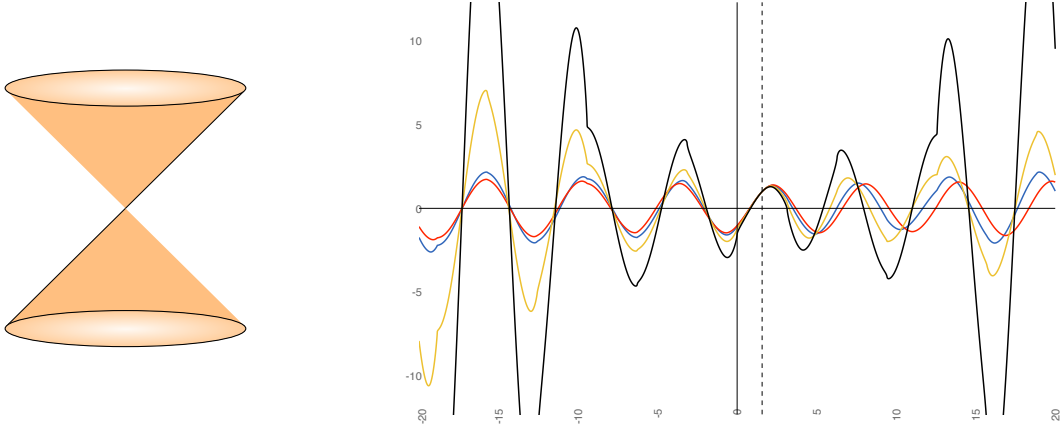
\begin{figure}
\hspace{-5cm}
\begin{minipage}{0.3\linewidth}
\begin{tikzpicture}[scale=0.8, font=\tiny]

 \fill[orange,opacity=0.5] (0,0) --  (2,1.98) -- (-2,1.98) -- (0,0) -- cycle;
  \fill[orange,opacity=0.5, rotate=180] (0,0) --  (2,1.98) -- (-2,1.98) -- (0,0) -- cycle; 

 \draw[shorten >= 2pt] (0,0) -- (2.02,2);
  \draw[shorten >= 2pt, rotate=180] (0,0) -- (2.02,2);

 \filldraw[draw=black,outer color=orange!40,inner color=orange!5] (0,2) circle[x radius=2,y radius=0.3];
 \filldraw[draw=black,outer color=orange!40,inner color=orange!5, rotate=180] (0,2) circle[x radius=2,y radius=0.3];
\end{tikzpicture}
\end{minipage}
\begin{minipage}{0.3\linewidth}
\includegraphics[scale=0.35]{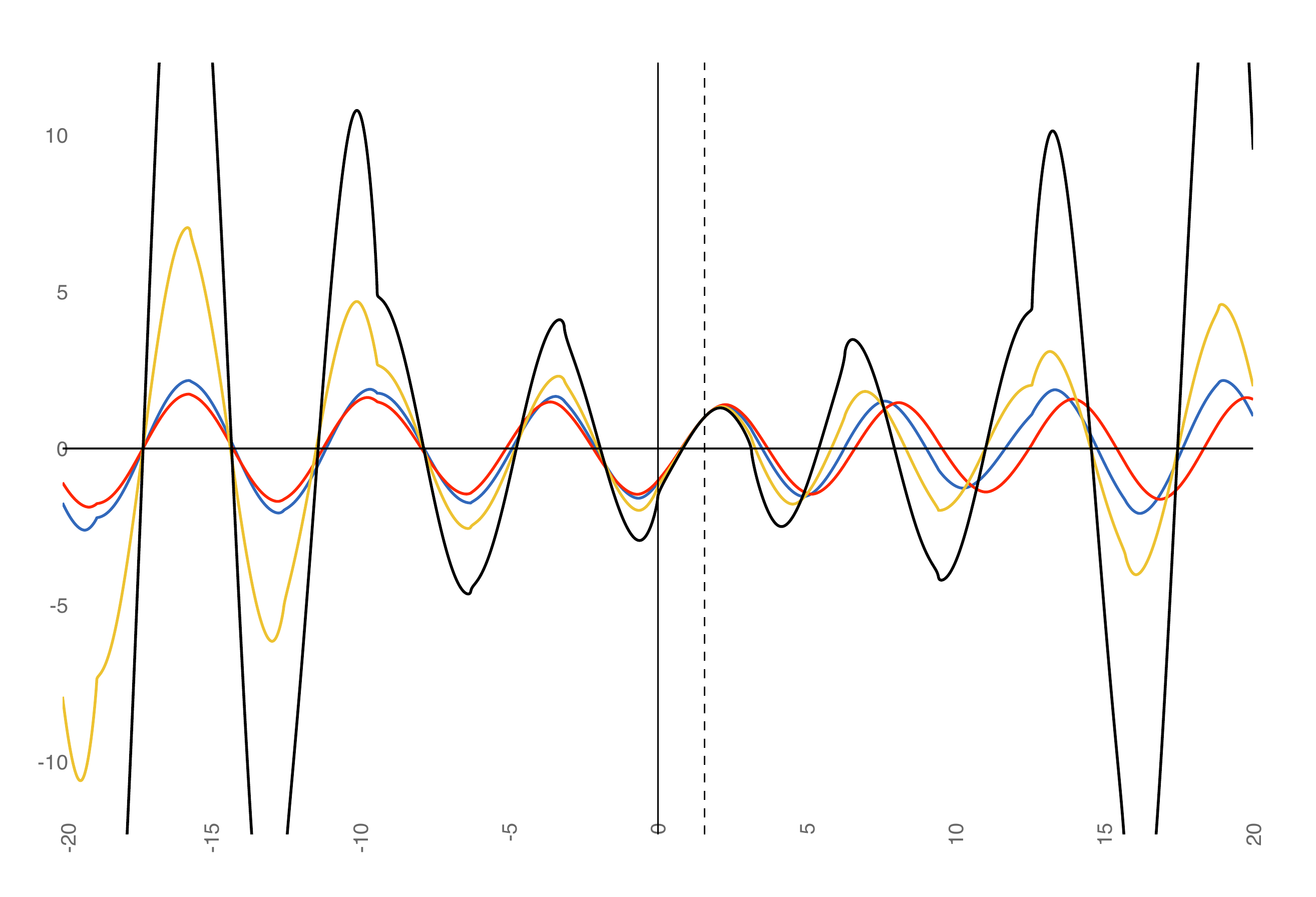}
\end{minipage}
\caption{The double cone (left). The orange region is 
$\{u=0\}$. Some solutions of \eqref{Legendre} with $f(\pi/2)=1, \dot f(\pi/2)=1$ (right): black $p=4.3$, yellow $p=10$, blue $p=29$, red $p=100$. The dashed line is $x=\pi/2$. }
\end{figure}

When $p=2$ it is shown in \cite{AC}  that $u$ is not a minimizer.
Note that the zero set of $u$ has non-trivial Lebesgue density since the free boundary is a cone.

\subsection{Reduction or 1st order ode}

Let us introduce the function $u$ so that $\dot f=u f$ then substituting this into \eqref{Legendre}
we get the following first order equation 
for $u$ 
\begin{eqnarray*}
 \dot u+u^2+(1+u\cot\theta)\frac{1+u^2}{1+(p-1)u^2}+1=0.
\end{eqnarray*}
 In Figure 2 four solutions to \eqref{Legendre} are illustrated for some values of $p$.

\subsection{Stability inequality}
The main result of this section is  the following 
\begin{prop}
Let $u=\displaystyle\rho \max\left(\frac{f(\theta)}{\dot f(\theta_0)}, 0\right)$ where 
$f$ solves \eqref{Legendre} and $\theta_0$ is the only zero of $f$ in $(0, \pi/2].$
Then $u$ is not a minimizer. 
\end{prop}
\begin{proof}
Recall \eqref{eq:xxxx} and consider 
\begin{eqnarray}\label{eq:2nd-var-00}
I= \int_{\po u} |\na u|^{p-2}\left\{|\na \psi|^2+(p-2)\frac{(\na u\na \psi)^2}{|\na u|^2}\right\}
 -
 \int_\Gamma H \psi^2. 
\end{eqnarray}
Note that 
\begin{equation}\label{blyaaa-2}
I\le (p-1)\int|\na u|^{p-2}|\na \psi|^2-
 \int_\Gamma H \psi^2.
\end{equation}
Let us take the truncated fundamental solution  
\[
\eta_\delta(x)=
\left\{
\begin{array}{lll}
\frac1{|x|} & \mbox{if}\  |x|>\delta,\\
\frac1\delta & \mbox{if}\ |x|\le \delta, 
\end{array}
\right.
\]
and consider $\psi_{R, \delta}(x)=\eta_\delta(x)\xi(\frac {|x|}R)$,  where $0\le \xi\le 1$ is a standard cut-off function such that 
$\xi(x)=1$ in $B_{1/2}$, $\xi\in C^\infty_0(B_1)$.  

Observe that 
\[
\begin{split}
\lim_{R\to \infty}\int \left| \nabla u\right| ^{p-2}\left| \nabla \psi_{R, \delta} \right| ^{2}
&=
\int_{|x|>\delta}|\na u |^{p-2}\left|\na \left(\frac{1}{|x|}\right)\right|^2\\
&\le 
C\|\na u\|_\infty^{p-2}\frac1\delta.
\end{split}
\]

Denoting $\psi_\delta(x) =\eta_\delta(x)$ and using the  divergence theorem 
we see that 
\begin{eqnarray*}
\int \left| \nabla u\right| ^{p-2}\left| \nabla \psi_\delta \right| ^{2}
&=&
\int _{\Gamma }\psi_\delta\p_\nu \psi_\delta
-
\int \psi_\delta \div  \left( |\nabla u |^{p-2} \nabla \psi_\delta \right) \\
&=&
\int _{\Gamma }\psi_\delta \p_\nu\psi_\delta
\end{eqnarray*}
because 
\begin{equation}\label{eq:fundam}
\div\left(|\na u|^{p-2}\na(\frac{1}{|x|})\right)=0, \quad \rho=|x|>\delta.
\end{equation}
We can see this from the equation 
\begin{eqnarray*}
\div\left(|\na u|^{p-2}\na(\frac{1}{|x|})\right)
=
\frac1{\rho\sin\theta}
\left\{
\frac\p{\p\rho}\left(\rho^2\sin\theta g(\theta)\p_\theta\psi_\delta\right)+\frac \p{\p\theta}\left(\rho\sin\theta\frac1\rho g(\theta)\p_\theta\psi_\delta\right)
\right\}, 
\end{eqnarray*}
where $g(\theta)= \left[f^2(\theta)+\dot f^2(\theta)\right]^{\frac{p-2}2}$.
Since in $\rho>\delta$
\[
\frac\p{\p\rho}\left(\rho^2\sin\theta g(\theta)\p_\theta\psi_\delta\right)
=
-\sin\theta g(\theta)\frac\p{\p\rho}\left(\rho^2 \frac1{\rho^2}\right)=0
\]
and,  moreover,  $\p_\theta\psi_\delta=0$ by definition of $\psi_\delta $, implying that  
\[
\frac \p{\p\theta}\left(\rho\sin\theta\frac1\rho g(\theta)\p_\theta\psi_\delta\right)=0.
\]
Hence \eqref{eq:fundam} is satisfied. 
Thus returning to \eqref{blyaaa-2} we conclude 

\[
\begin{split}
I
&\le
(p-1)\int _{\Gamma }\psi_\delta \p _{\nu}\psi_\delta
- 
 \int_\Gamma H \psi^2_\delta\\
 &=
 (p-1)\int _{\Gamma\cap \{|x|>\delta\} }\dfrac {1}{\rho^{2}}\frac{f}{\sqrt{f^2+(\dot f)^2}}
 - 
 \int_\Gamma H \psi^2_\delta, 
\end{split}
\]
because 
\[
\begin{split}
\p_\nu\psi _\delta
&=
-\dfrac {\nabla u}{\left| \nabla u\right| }\nabla \psi _\delta
=
-\dfrac {1}{\sqrt {f^{2}+\left( f'\right) ^{2}}}\left( -\dfrac {1}{\rho ^{2}},0,0\right) \cdot \left( f,0,\dot f\right) \\
&=\dfrac {1}{\rho^{2}}\frac{ f}{\sqrt{f^2+(\dot f)^2}}. 
\end{split}
\]
But $\theta=\theta_0$ on $\Gamma\setminus B_{\delta}$ and $f(\theta_0)=0, \dot f(\theta_0)\not=0$, so  we have 
\[
\dfrac {1}{\rho^{2}}\frac{f(\theta_0)}{\sqrt{f^2(\theta_0)+(\dot f(\theta_0))^2}}=0.
\]
Summarizing we infer 
\[
I\le - \int_\Gamma H \psi^2_\delta=-\int_\delta^\infty \frac{d\rho}{\rho^2}\int_{\S^2\cap \Gamma}\kappa<0, 
\]
where $\kappa>0$ is the curvature of the circle $\fb u\cap \p B_1$. 
\end{proof}


\end{document}